\documentclass[12pt,fleqn]{article}
\usepackage{latexsym}
\usepackage{amsmath}
\usepackage{amssymb}
\usepackage{epsfig}
\usepackage{amsfonts}

\begin{document}
\newtheorem{definition}{Definition}[section]
\newtheorem{theorem}[definition]{Theorem}
\newtheorem{lemma}[definition]{Lemma}
\newtheorem{proposition}[definition]{Proposition}
\newtheorem{examples}[definition]{Examples}
\newtheorem{corollary}[definition]{Corollary}
\def\square{\Box}
\newtheorem{remark}[definition]{Remark}
\newtheorem{remarks}[definition]{Remarks}
\newtheorem{exercise}[definition]{Exercise}
\newtheorem{example}[definition]{Example}
\newtheorem{observation}[definition]{Observation}
\newtheorem{observations}[definition]{Observations}
\newtheorem{algorithm}[definition]{Algorithm}
\newtheorem{criterion}[definition]{Criterion}
\newtheorem{algcrit}[definition]{Algorithm and criterion}

\newenvironment{prf}[1]{\trivlist
\item[\hskip \labelsep{\it
#1.\hspace*{.3em}}]}{~\hspace{\fill}~$\square$\endtrivlist}
\newenvironment{proof}{\begin{prf}{Proof}}{\end{prf}}

\title{Real and p-adic Picard--Vessiot fields}
\author{Teresa Crespo, Zbigniew Hajto and Marius van der Put }
\date{}
\maketitle

\begin{abstract} We consider  differential modules over  real and p-adic  differential fields $K$ such that its field of constants $k$ is real closed (
resp., p-adically closed).  Using P.~Deligne's work on Tannakian categories and a result of J.-P.~Serre on Galois cohomology, a purely algebraic proof of the existence and unicity of  real (resp., p-adic) Picard-Vessiot fields is obtained. The inverse problem for real forms of a semi-simple  group is treated. Some examples illustrate the relations between differential modules, Picard--Vessiot fields and real forms of a linear algebraic  group. 
     \end{abstract}

\section{Introduction}

 We are grateful to J.-P.~Serre for providing us with a proof of the following statement on Galois cohomology.
\begin{theorem} Let $k\subset K$ denote fields of characteristic zero such that:\\
{\rm (i)}. For every smooth variety $V$ of finite type over $k$, one has that $V(K)\neq \emptyset $ implies $V(k)\neq \emptyset$.
{\rm (ii)}. The natural map $Gal(\overline{K}/K)\rightarrow Gal(\overline{k}/k)$ is bijective.\\
Let $G$ be any linear algebraic group over $k$, then the natural map\\
$H^1(k,G(\overline{k}))\rightarrow H^1(K,G(\overline{K}))$ is bijective.  
\end{theorem}
We will apply this result for the case that $k,K$ are both real closed fields
or both p-adically closed fields. For notational convenience a field $F$
is called {\it (formally) p-adic} if there is a valuation ring $O\subset F$ (with field of fractions $F$)  such that $pO$ is the maximal ideal of $O$ and $O/pO=\mathbb{F}_p$. The field $F$ is called {\it p-adically closed} 
if moreover no proper algebraic extension of $F$ is again a p-adic field.
What we call ``p-adic'' is called ``p-adic of rank one'' in
 \cite{PR}.  Higher rank p-adic fields can be treated in the same way.

We note that this theorem has also as a consequence that the classification of semi-simple Lie algebras over a real closed field (or a p-adically closed field) does not depend on the choice of that field.

In the sequel we will often write statements and proofs only for the real case and mention that the p-adic case is similar.\\

$K$ denotes a real differential field with field of constants $k$. We will always suppose that $k\neq K$ and
that $k$ is a real closed field. Let $M$ denote a differential module over $K$ of dimension $d$,
represented by a matrix differential equation $y'=Ay$ where $A$ is a $d\times d$-matrix with entries
in $K$. A {\it Picard--Vessiot field} $L$ for $M/K$ is a field extension of $K$ such that:\\
(a) $L$ is equipped with a differentiation extending the one of $K$,\\
(b)  $M$ has a full space of solutions over $L$, i.e., there exists an invertible $d\times d$-matrix $F$ (called
a fundamental matrix) with entries in $L$ satisfying $F'=AF$,\\
(c) $L$ is (as a field) generated over $K$ by the entries of $F$,\\
(d) the field of constants of $L$ is again $k$.\\

A {\it real Picard--Vessiot field} $L$ for $M/K$ is a
Picard--Vessiot field which is also a real field. For a p-adic differential
field $K$ we always suppose that its field of constants is p-adically closed.
The definition of a  {\it p-adic Picard--Vessiot field} is similar.
The main result of this paper is:
\begin{theorem} $K\supset k$ as above. Let $M/K$ be a differential module. \\
{\rm (1). Existence.}  There exists a real (resp., p-adic) Picard--Vessiot extension for $M/K$.\\
{\rm (2). Unicity for the real case}. Let $L_1,L_2$ denote two real Picard--Vessiot extensions for $M/K$. Suppose that $L_1$ and $L_2$
have total orderings which induce the same total ordering on $K$. Then there exists a $K$-linear isomorphism
$\phi : L_1\rightarrow L_2$ of differential fields.\\
{\rm (3). Unicity for the p-adic case}. Let $L_1,L_2$ denote two p-adic
Picard--Vessiot extensions for $M/K$. Suppose that $L_1$ and $L_2$
have $p$-adic closures $L_1^+$ and $L_2^+$ such that the 
$p$-adic valuations of $L_1^+$ and $L_2^+$ induce the same
$p$-adic valuation
on $K$ and such that   $K\cap (L_1^+)^n=K\cap (L_2^+)^n$ for every integer $n\geq 2$ (where $F^n:=\{f^n| f\in F\}$). Then there exists a $K$-linear isomorphism $\phi : L_1\rightarrow L_2$ of differential fields. \end{theorem}
$ \ $
\begin{remarks} {\rm $\ $ \\
(1).  It seems to be folklore that in the case the field of constants is $\mathbb{R}$, a Picard--Vessiot field (real or not) need not exist (compare \cite{C-S}, Remark 2.2). This is due to a mistaken interpretation of an example of A.~Seidenberg \cite{Sei}.   \\

\noindent 
(2). Consider part (2) of Theorem 1.2. Suppose that an isomorphism $\phi :L_1\rightarrow L_2$ exists.  Choose a total ordering of $L_1$ and define the total ordering of $L_2$ to be induced
by $\phi$. Then $L_1$ and $L_2$ induce the same total ordering on $K$. Therefore the condition in part (2)  of Thm. 1.2 is necessary.

If $K$ happens to be real closed, then the assumption in part (2) of Thm. 1.2 is superfluous since $K$ has a unique total ordering. On the other hand, consider the example $K=k(z)$ with differentiation $'=\frac{d}{dz}$ and the equation
$y'=\frac{1}{2z}y$. Let $L_1=K(t_1)$ with $t_1^2=z$ and $L_2=K(t_2)$ with $t_2^2=-z$. Both fields are real
Picard--Vessiot fields for this equation. They are not isomorphic as (differential) field extensions of $K$.  We observe  that 
$z$ is positive for any total ordering of $L_1$ and $z$ is negative for any total ordering of $L_2$.\\

\noindent 
(3). Consider part (3) of Theorem 1.2. As in part (2), the condition is necessary. It is superfluous if $K$ is p-adically closed.  Consider the equation $y'=\frac{1}{2z}y$ over the differential field $\mathbb{Q}_p(z)$
with $z'=1$. Then $L_j=K(t_j),\ j=1,2,3,4$ and 
$t_1^2=z,\ t_2^2=pz,\ t_3^2=az,\ t_4^2=paz$, where the image of $a\in \mathbb{Z}_p^*$ in $\mathbb{F}_p^*$ is not a square, are non isomorphic p-adic Picard-Vessiot fields.  Let $L_j^+$ denote a $p$-adic closure of $L_j$. It is possible to choose the $L_j^+$ such that the 
$p$-adic valuations induce the same $p$-adic valuation on $K$. However
the sets $K\cap (L_j^+)^2$ are  clearly distinct. }\hfill $\square$ \\ \end{remarks}

The proof of Theorem 1.2 uses Tannakian categories as presented in
\cite{D-M} and  P.~Deligne's fundamental paper \cite{De}. We adopt
much of the notation of \cite{De}. Let $<M>_\otimes $ denote the
Tannakian category generated by the differential module $M$. The
forgetful functor $\rho : <M>_\otimes \rightarrow vect(K)$
associates to a differential module $N\in <M>_\otimes $ the
finite dimensional $K$-vector space $N$. 

 \begin{lemma} There exists a fibre functor
 $\omega : <M>_\otimes \rightarrow vect(k)$.
 \end{lemma}
\begin{proof} We follow the reasoning of \cite{De}, 6.20 Corollaire.
The field $K$ contains a finitely generated $k$-subalgebra $R$, which 
is invariant under differentiation, such that there exists a projective, finitely generated
$R$-module $M_0\subset M$, with the properties: $M_0$ is invariant under the operator
$\partial $ of $M$ and the canonical map $K\otimes _RM_0\rightarrow M$ is an isomorphism of differential modules.   Let $<M_0>_\otimes$ denote the Tannakian category of the projective differential modules over $R$, generated by $M_0$. Moreover, $R$ can be chosen
 such that the functor $K\otimes _R*:<M_0>_\otimes \rightarrow <M>_\otimes $ is an equivalence of Tannakian categories. Choose a
maximal ideal $\underline{m}$ of $R$.  Then $N\mapsto R/\underline{m}\otimes N$
is a fibre functor $<M_0>_\otimes \rightarrow vect(R/\underline{m})$ and induces a fibre
functor $<M>_\otimes \rightarrow vect(R/\underline{m})$.

 In our special case $K$ is a real field (resp., p-adic field) and therefore $R$ is a real 
(resp., p-adic) algebra, finitely generated over a real closed (resp., p-adically closed) field $k$. There exists  $\underline{m}$ such that $R/\underline{m}=k$ (see \cite{La, PR}).   \end{proof}

 {\it Now we recall some results of}  \cite{De}, \S 9. 
 The functor $\underline{Aut}^\otimes(\omega)$ is represented by a linear
 algebraic group $G$ over $k$.  By Proposition 9.3, the functor $\underline{Isom}^\otimes_K(K\otimes \omega ,\rho)$ is represented by a torsor $P$ over $G_K:=K\times _kG$.   This torsor is affine, irreducible and its coordinate ring $O(P)$ has a natural differentiation extending the differentiation of $K$. Moreover, the field of fractions $K(P)$ of $O(P)$ is  a Picard--Vessiot field for $M/K$ and $G$ identifies with the group
 of the $K$-linear differential automorphisms of $K(P)$.

 On the other hand, let $L$ be a Picard--Vessiot for $M/K$. Define the fibre functor
 $\omega _L: <M>_\otimes \rightarrow vect(k)$ by $\omega _L(N)=\ker (\partial : L\otimes _KN\rightarrow L\otimes_K N)$.
 Then $\omega _L$ produces a Picard--Vessiot field $L'$ which is isomorphic to $L$ as differential field extension
  of $K$. The conclusion is:\\

 \begin{proposition}[{\rm \cite{De},\S 9}]
 The above constructions yield a bijection between the (isomorphy classes of) fibre functors
 $\omega : <M>_\otimes \rightarrow vect(k)$ and the (isomorphy classes of) Picard--Vessiot fields $L$ for $M/K$.
  \end{proposition}

The following result will also be useful.

\begin{proposition}[{\rm \cite{D-M}, Thm. 3.2}]  Let $\omega :<M>_\otimes \rightarrow vect(k)$ be a fibre functor and
$G=\underline{Aut}^\otimes_k(\omega)$.\\
\noindent {\rm (a)}. For any field $F\supset k$ and any fibre functor $\eta :  <M>_\otimes \rightarrow vect(F)$,
 the functor $\underline{Isom}^\otimes _F(F\otimes \omega ,\eta)$ is representable by a torsor over
$G_F=F\times _kG$.\\
{\rm (b)}. The map $\eta\mapsto  \underline{Isom}^\otimes _F(F\otimes \omega ,\eta)$ is a bijection between
the (isomorphy classes of) fibre functors  $\eta :  <M>_\otimes \rightarrow vect(F)$ and the (isomorphy classes of)
$G_F$-torsors. \end{proposition}

The final ingredient in the proof of Theorem 1.2 is:

   \begin{theorem} Suppose that $K$ is real closed (resp., p-adically closed). Let $L$ be a Picard--Vessiot field for $M/K$. Then $L$ is a real
field  (resp., a p-adic field) if and only if the torsor $\underline{Isom}_K^\otimes (K\otimes \omega _L,\rho)$ is trivial. \end{theorem}

\section{The proof of Theorem 1.2}
\subsection{Reduction to $K$ is real  (resp., p-adically) closed}
For notational convenience, a differential module $M/K$ is represented by a
linear differential equation $\mathcal{L}(y):=y^{(d)}+a_{d-1}y^{(d-1)}+\cdots +a_1y^{(1)}+a_0y=0$.
A Picard--Vessiot field $L$ for $\mathcal{L}$ has the properties: 
$k$ is the field of constants of $L$, the {\it solution space}
$V=\{v\in L|\ \mathcal{L}(v)=0\}$ is a $k$-linear space of dimension $d$ and $L$ is generated over the field $K$ by $V$ and all the derivatives of the elements in $V$. One writes $L=K<V>$ for this last property.
\begin{lemma} Let $\tilde{K}\supset K$ be an extension of real (resp., p-adic) differential fields such that the field of constants of $\tilde{K}$ is $k$. Suppose that $\tilde{K}\otimes M$ has a real (resp., p-adic) Picard--Vessiot field
$\tilde{L}$, then $M$ has a real (resp., p-adic) Picard--Vessiot field.
\end{lemma} 
\begin{proof} Let $V\subset \tilde{L}$ denote the solution space of  $\tilde{K}\otimes M$. Then the field $L=K<V>\subset \tilde{L}$ is clearly a real
(resp., p-adic) Picard--Vessiot field for $M$. \end{proof}

\begin{lemma} Let $L_1,L_2$ be two real Picard--Vessiot fields for $M$ over the real differential field $K$. Suppose that
 $L_1$ and $L_2$ have total orderings extending a total ordering $\tau$ on $K$. Let $K^r\supset K$ be the real closure of $K$ inducing the total ordering $\tau$. Then:\\

 The fields $L_1,L_2$ induce Picard--Vessiot fields
 $\tilde{L}_1,\tilde{L}_2$ for $K^r\otimes M$ over $K^r$. These fields
 are isomorphic as differential field extensions of $K^r$ if and only if
 $L_1$ and $L_2$ are isomorphic as differential field extensions of $K$.
\end{lemma}
\begin{proof} Let, for $j=1,2$, $\tau_j$ be a total ordering on $L_j$ inducing $\tau$ on $K$ and
let $L_j^r$ be the real closure of $L_j$ which induces the ordering $\tau_j$. The algebraic closure
$K_j$ of $K$ in $L_j^r$ is real closed. Since $\tau_j$ induces $\tau$, there exists a $K$-linear
isomorphism $\phi _j:K^r\rightarrow K_j$. This isomorphism is unique since the only $K$-linear
automorphism of $K^r$ is the identity. We will identify $K_j$ with $K^r$.

 Let $V_j\subset L_j$ denote the solution space of $M$. Then, for $j=1,2$, the field
 $\tilde{L}_j:=K^r<V_j>\subset L_j^r$ is a real Picard--Vessiot field for $K^r\otimes M$.

   Assume the existence of
 a $K^r$-linear differential isomorphism\linebreak $\psi :K^r<V_1>\rightarrow K^r<V_2>$.
 Clearly $\psi (V_1)=V_2$ and $\psi$ induces therefore a $K$-linear differential isomorphism
 $L_1=K<V_1>\rightarrow L_2=K<V_2>$.

 On the other hand, an isomorphism $\phi : L_1\rightarrow L_2$ (of differential field extensions of $K$) extends
 to an isomorphism $\tilde{\phi}: L_1^r\rightarrow L_2^r$. Clearly $\tilde{\phi}$ maps $\tilde{L}_1$ to $\tilde{L}_2$.
 \end{proof}

{\it The p-adic version of Lemma 2.2 reads as follows}.\\
Let $L_1,L_2$ be two p-adic Picard--Vessiot fields for the differential module $M$ over the p-adic differential field $K$.  Let $L_1^+$
and $L_2^+$ denote $p$-adic closures of $L_1$ and $L_2$ satisfying 
the condition of part (3) of Theorem 1.2. 

 For $j=1,2$, the algebraic closure $K_j$ of $K$ in $L_j^+$ is a $p$-adic closure of $K$, according
to \cite{PR}, Theorem 3.4. Further $K\cap K_1^n=K\cap K_2^n$ for every
integer $n\geq 2$ since $K\cap (L_1^+)^n=K\cap (L_2^+)^n$ holds.
By \cite{PR}, Corollary 3.11, there is a $K$-linear isomorphism 
of the $p$-adic fields $K_1\rightarrow K_2$. Now we identify $K_1$ and
$K_2$ and denote this field by $\tilde{K}$.  Then:

The fields $L_1,L_2$ induce Picard--Vessiot fields  $\tilde{L}_1,\tilde{L}_2$ for $\tilde{K}\otimes M$ over $\tilde{K}$.  As in the proof 
Lemma 2.2, one shows that  $\tilde{L}_1$ and $\tilde{L}_2$
 are isomorphic as differential field extensions of $\tilde{K}$ if and only if
 $L_1$ and $L_2$ are isomorphic as differential field extensions of $K$.\\

{\it We conclude that it suffices to prove
Theorem 1.2 for the case that $K$ is real (resp., p-adically) closed}.

\subsection{Proof of  the unicity}

 \begin{theorem}[{\bf  =1.7}]  Suppose that $K$ is real (resp., p-adically)  closed. Let $L$ be a Picard--Vessiot field for a differential module
 $M/K$.  Then $L$ is a real (resp., p-adic) field if and only if the torsor $\underline{Isom}^\otimes_K(K\otimes \omega_L, \rho)$ is
 trivial. \end{theorem}
 \begin{proof}  We write the proof for the real case. The p-adic case is similar.

The coordinate ring of the affine torsor  $\underline{Isom}^\otimes_K(K\otimes \omega_L, \rho)$ is denoted by $R$. We recall that  $L$ is the field of fractions of $R$.\\

If $L$ is a real Picard--Vessiot field, then $R\subset L$ is a finitely generated real $K$-algebra. From the real
Nullstellensatz and the assumption that $K$ is real closed it follows that there exists a $K$-linear homomorphism
$\phi :R\rightarrow K$ with $\phi (1)=1$. The torsor $Spec(R)$ has a $K$-valued point and is therefore trivial.  \\

If the torsor $Spec(R)$ is trivial, then the affine variety $Spec(R)$ has 
a $K$-valued point.  It follows that the Picard--Vessiot field $L$, which is the function field  of this variety, is real (see for instance \cite{PR}, \S 1, Theorem 2). \end{proof}

{\it In proving the unicity}, we restrict to the real case. According to \S 2.1 we
may assume that $K$ is real closed.  Let $L_1,L_2$ denote two real Picard-Vessiot fields for a differential module $M/K$. We will prove
that $L_1$ and $L_2$ are isomorphic as differential extension fields of $K$.
\begin{proof}
 Write $\omega _j=\omega _{L_j}:<M>_\otimes \rightarrow vect(k)$ for the corresponding
fibre functors. Put $G=\underline{Aut}^\otimes_k(\omega _1)$. Then $\underline{Isom}^\otimes _k(\omega _1,\omega _2)$ is a $G$-torsor over $k$ corresponding to an element $\xi \in H^1(k,G(\overline{k}))$.

The $G_K$-torsor  $\underline{Isom}^\otimes _K(K\otimes \omega
_1,K\otimes \omega _2)$ corresponds to an  element
 $\eta \in H^1(K,G(\overline{K}))$. This element is the image of $\xi$ under the natural map from 
$H^1(k,G(\overline{k}))$ to $H^1(K,G(\overline{K}))$,
induced by the inclusion $G(\overline{k})\subset G(\overline{K})$ and the observation that $Gal(\overline{K}/K)=Gal(\overline{k}/k)$).   Since $L_j$ is real, the torsor
$\underline{Isom}^\otimes _K(K\otimes \omega _j,\rho )$ is trivial for $j=1,2$, by Theorem 1.7. Thus there exists
 isomorphisms $\alpha _j :K\otimes \omega _j\rightarrow \rho$ for $j=1,2$. The isomorphism
 $\alpha _2^{-1}\circ \alpha _1:K\otimes \omega _1\rightarrow K\otimes \omega _2$ implies that $\eta$ is trivial.

Let $\xi$ be represented by the 1-cocycle $c$ with values in $G(\overline{k})$. Since its image
in $H^1(K,G(\overline{K}))$ is trivial, there is an element $h\in G(\overline{K})$ such that $c(\alpha )=h^{-1}\alpha (h)$ for all $\alpha \in Gal(\overline{K}/K)=Gal(\overline{k}/k)$.

 There exists a finitely generated $k$-algebra $B\subset K$ with 
$h\in G(\overline{k}B)$. Since $B$ is real and $k$ is real closed, there exists a 
$k$-linear homomorphism $\phi :B\rightarrow k$ with $\phi (1)=1$
(\cite{PR,La}). Further $\phi$ extends
to a $\overline{k}$-linear homomorphism $\overline{k}B\rightarrow \overline{k}$, commuting with the actions
of $Gal(\overline{K}/K)=Gal(\overline{k}/k)$. 
  Applying $\phi$ to the identity $c(\alpha)=h^{-1}\alpha(h)$ one obtains
 $c(\alpha )=\phi (h)^{-1}\alpha (\phi (h))$.
  Thus $c$ is a trivial 1-cocycle and there is an isomorphism $\omega _1\rightarrow \omega _2$. Hence  $L_1$ and $L_2$ are isomorphic as differential field extensions of $K$.\end{proof}

\subsection{Proof of Theorem 1.1}
Let $k\subset K$ satisfy (i) and (ii). The final part of \S 2.2 proves the injectivity of $H^1(k, G(\overline{k}))\rightarrow H^1(K, G(\overline{K}))$. For notational convenience we write this as $H^1(k, G)\rightarrow H^1(K, G)$.
 {\it Here we reproduce J.-P.~Serre's proof of the surjectivity of this map,
which he communicated to us by an email message on 04-07-2013}.
\begin{proof}  
(1). Let $U$ be the unipotent radical of  $G$.  The map  $H^1(k,G) \rightarrow H^1(k,G/U)$ is bijective. One can find the easy proof of this in 
lemma 7.3 of \cite{G-MB}. 
 Since  $H^1(K,G) \rightarrow H^1(K,G/U)$ is also bijective,  we may
divide by  $U$, i.e. we may assume that the neutral component  $G^o$  of  $G$  is reductive.\\
(2). Consider a commutative group  $C$ over $k$. It is given that the natural map $Gal(\overline{K}/K)\rightarrow Gal(\overline{k}/k)$ is a bijection. Hence there are natural maps  $H^n(k,C) \rightarrow H^n(K,C)$ for all $n$. For every $n > 0$ (we only need $n = 1,2$) these maps are bijective. Indeed, the commutative group  $C(\overline{K})/C(\overline{k})$ is torsion free and divisible and so it has trivial Galois cohomology.\\
(3).  Let  $T$  be a maximal torus  of  $G$, and let $N$  be its normalizer. By a result of T.A.~Springer (lemma 6 of III.4.3 \cite{Se}), the map  $H^1(K,N) \rightarrow H^1(K,G)$  is surjective. Hence it will be enough to prove surjectivity for  $N$.\\
(4). After replacing $G$ by $N$,  we have an exact sequence  $1 \rightarrow C \rightarrow G \rightarrow F \rightarrow 1$, where  $C$  is a torus
and  $F$ a finite group. This gives us a diagram for the  $H^1$:

$$\begin{array}{ccccccc} 1  & \rightarrow & H^1(k,C) & \rightarrow & H^1(k,G) & \rightarrow & H^1(k,F) \\
&& \downarrow && \downarrow && \downarrow \\  1  & \rightarrow &   H^1(K,C)  & \rightarrow & H^1(K,G)  & \rightarrow & H^1(K,F)\end{array}
$$
The map $H^1(k,F)\rightarrow H^1(K,F)$ is bijective since $F$ is finite and
$Gal(\overline{K}/K)=Gal(\overline{k}/k)$ and we identify the two sets.
 Let  $x$  be an element of  $H^1(K,G)$  and let  $y$  be its image in  $H^1(K,F)$. 
 Thus we view   $y$   as an element of $H^1(k,F)$.

\noindent {\bf Claim.} The element  $y$  belongs to the image of  $H^1(k,G) \rightarrow H^1(k,F)$.

\noindent {\it Proof of the claim}. We use prop.41 of I.5.6 \cite{Se}. It tells us that  $y$  belongs to that image if and only if its coboundary $\Delta(k,y)$ 
is  $0$; this coboundary belongs to  $H^2(k,C_y)$, where  $C_y$ is the Galois twist of  $C$  defined by a cocycle of the class  $y$.
This argument also applies over  $K$, so that we have a class  $\Delta(K,y)$ in  $H^2(K,C_y)$, and it is clear that  $\Delta(K,y)$
is the image of $\Delta(k,y)$ under the map  $H^2(k,C_y) \rightarrow H^2(K,C_y)$. But, since $y$  comes from  $x$, we have
$\Delta(K,y) = 0$, hence $\Delta(k,y) = 0$, by  (2)  above, applied to the commutative group  $C_y$.\\
(5).  End of proof. By a little diagram chasing,  the Claim and (2), one shows that  $x$  belongs to the image of  $H^1(k,G)$. \end{proof}

\subsection{Proof of the existence}
We present the proof for the real case. The p-adic case is similar.
One considers a differential module $M$ over a real closed differential field $K$.
 We fix a fibre functor $\omega _0:<M>_\otimes \rightarrow vect(k)$ and write $G_0:=Aut^\otimes (\omega _0)$.
Further $G_\rho:=Aut^\otimes (\rho )$, where $\rho :<M>_\otimes \rightarrow vect(K)$ is the forgetful functor.\\
We recall from Proposition 1.5, that $H^1(k,G_0(\overline{k}))$ can be identified
with the set of the fibre functors $\omega :<M>_\otimes \rightarrow vect(k)$ and that 
 $ H^1(K,G_\rho(\overline{K}))$ can be identified with the set of the right $G_\rho$-torsors.\\

Thus the map $\omega \mapsto \underline{Isom}(K\otimes \omega ,\rho)$
from fibre functors to right $G_\rho$-torsors can be seen as a map 
$H^1(k,G_0(\overline{k}))
\rightarrow H^1(K,G_\rho(\overline{K}))$. We want to show that the trivial element of
$ H^1(K,G_\rho(\overline{K}))$ is in the image, because this means that some 
fibre functor $\omega$ yields a trivial torsor, or translated, $L_\omega$ is a real Picard--Vessiot field.
The above map factors as 
\[H^1(k,G_0(\overline{k}))\stackrel{natG_0}{\rightarrow} H^1(K,G_0(\overline{K}))
\stackrel{composition}{\rightarrow}H^1(K,G_\rho(\overline{K})).\]
Here $natG_0$ denotes the natural map and the map ``$composition$'' is defined as follows. An element in $ H^1(K,G_0(\overline{K}))$
is a right $K\times _kG_{0}$-torsor. One can compose with $\underline{Isom}^\otimes (K\otimes \omega _0,\rho)$ which is a left $K\otimes _kG_0$-torsor and a right $G_\rho$-torsor. The result is a
right $G_\rho$-torsor and thus an element in $H^1(K,G_\rho(\overline{K}))$.
The map ``$composition$'' is clearly bijective.  According to Theorem 1.1, the map
$natG_0$ is bijective. This finishes the proof of  the existence.

\section{Comments and Examples}
        $\ $ \\
\indent  The proof of the unicity and existence of  real (resp., p-adic) Picard--Vessiot fields uses almost exclusively properties of Tannakian
categories and Galois cohomology of linear algebraic groups. This implies that the proof remains valid for other types of equations, such as:\\
(a). linear  partial  differential equations, like $\frac{\partial}{\partial z_j}Y=A_jY$ for $j=1,\dots ,n$,\\
(b). linear ordinary difference equations, like $Y(z+1)=AY(z)$, \\
(c). linear $q$-difference equations, like $Y(qz)=AY(z)$, with $q\in \mathbb{R}^*$ (resp., $q\in \mathbb{Q}_p^*$).\\

\begin{observations} From Picard--Vessiot fields to differential Galois groups. \\ {\rm   
Let $K$ be a real closed differential field with field of constants $k$,  $M/K$ a differential module and $\omega :<M>_\otimes \rightarrow vect(k)$ a fibre functor. Let $L$ be the Picard--Vessiot field
corresponding to $\omega$ and $G$ the group of the differential automorphisms of $L/K$. Let
$H$ be the differential Galois group of $K(i)\otimes M$ over $K(i)$. We recall that $G$ is a form of $H$ over the field $k(i)$. Using the
identification $k(i)\times _kG=H$, one obtains on $H$ and on $Aut(H)$ a structure of algebraic group over $k$. Let $\{1,\sigma\}$ be the Galois group of $k(i)/k$. Then $H^1(\{1,\sigma \}, Aut(H))$ has a natural bijection to the set of forms of $H$ over $k$. Although the action of
$\sigma$ on $Aut(H)$ depends on $G$, this set does not depend on the choice of $G$.

Let $\eta: <M>_\otimes \rightarrow vect(k)$ be another fibre functor. Then $\eta$ is mapped, according to Proposition 1.5, to an element  $\xi (\eta) \in H^1(\{1,\sigma \}, G(k(i)))$ (and this induces a bijection between the set of the classes of $\eta$'s and this cohomology set). A 1-cocycle $c$ for the group $\{1,\sigma\}$ has the form $c(1)=1,\ c(\sigma )=a$ and
$a$ should satisfy $a\cdot \sigma (a)=1$ (and is thus determined by $a$).

 A 1-cocycle for
$\xi (\eta)$ can be made as follows. The fibre functor $\eta$ corresponds
to  a Picard--Vessiot field $L_\eta$. Both $L(i)$ and $L_\eta (i)$ are
Picard--Vessiot fields for $K(i)\otimes M$ over $K(i)$. Thus there 
exists a $K(i)$-linear differential isomorphism $\phi :L(i)\rightarrow 
L_\eta (i)$. On the field $L(i)$ we write $\tau$ for the conjugation given
by $\tau(i)=-i$ and $\tau$ is the identity on $L$. The similar conjugation
on $L_\eta(i)$ is denoted by $\tau_\eta$.
 Now $\tau_\eta \circ \phi\circ \tau :L(i)\rightarrow L_\eta (i)$ is another
$K(i)$-linear differential isomorphism. A 1-cocycle $c$ for $\xi (\eta)$
is now  $c(\sigma )=\phi ^{-1} \circ \tau_\eta \circ \phi\circ \tau $.\\

Let $G_\eta$ denote the group of the $K$-linear differential automorphisms of $L_\eta$. The group $G_\eta$ is a form of $G$ and
produces an element in $H^1(\{1,\sigma \},Aut(H))$ with $H=k(i)\times G$. We want to compute a 1-cocycle $C$ for this element. 
Define the isomorphism $\psi: k(i)\times G\rightarrow k(i)\times G_\eta$
of algebraic groups over $k(i)$, by $\psi (g)=\phi \circ g\circ \phi ^{-1}$.  
Define $\tau_G$, the `conjugation' on $k(i)\times G$, by the formula
$\tau_G(g)=\tau \circ g\circ \tau$ for the elements $g\in G(k(i))$.
Let $\tau_{G_\eta}$ be the similar conjugation on $k(i)\times G_\eta$.   Now $\tau_{G_{\eta}}\circ \psi \circ \tau_G:k(i)\times G\rightarrow
k(i)\times G_\eta$  is another isomorphism between the algebraic groups over $k(i)$. The 1-cocycle $C$ is given by
 $C(\sigma )=\psi ^{-1}\circ \tau_{G_{\eta}}\circ \psi \circ \tau_G$.
One observes that $C(\sigma )(g)=c(\sigma )gc(\sigma )^{-1}$.

The map, which associates to
$h\in G(k(i))$, the automorphism $g\mapsto hgh^{-1}$ of $G$, induces a map $H^1(\{1,\sigma\},G(k(i)))\rightarrow H^1(\{1,\sigma\},G/Z(G)(k(i)))
\rightarrow H^1(\{1,\sigma \}, Aut(H))$, denoted by 
$\xi(\eta ) \mapsto \tilde{\xi}(\eta)$. 
The forms corresponding to elements in the image of $H^1(\{1,\sigma\},G/Z(G)(k(i))) \rightarrow H^1(\{1,\sigma \}, Aut(H))$ are called `inner forms of $G$'. 
By \S1,  $\eta$ induces a Picard--Vessiot field and a form $G(\eta )$ of $H$.  Above we have verified (see  \cite{B} for a similar computation) that $G(\eta )$ is the inner form of $G$ corresponding to the element $\tilde{\xi}(\eta)$.  
For the  delicate theory of forms we refer to the informal manuscript
\cite{B} and the standard text \cite{Sp}. }\hfill $\square$ 
\end{observations}
\begin{examples} Picard--Vessiot fields and their groups of differential automorphisms.  {\rm
 {\it We continue with the notation and assumptions  of Observations} (3.1). \\ (1). Let $M/K, \omega, L,G$ be such that $G={\rm SL}(n)_{k}$. Since 
$H^1(\{1,\sigma \}, {\rm SL}(n)(k(i)))$ is trivial, $L$ is the unique Picard--Vessiot field and is a real field (because a real Picard--Vessiot field exists). 

The group ${\rm SL}(n)$ has non trivial forms. For instance, 
${\rm SU}(2)$ is an inner form of ${\rm SL}(2)_{\mathbb{R}}$. There are
examples, according to Proposition 3.2, of differential modules $M/K$ 
having a real Picard--Vessiot field $L$ with group of differential automorphisms of $L/K$ equal to ${\rm SU}(2)$.

From \cite{Sp}, 12.3.7 and 12.3.9 one concludes that 
$H^1(\{1,\sigma \},{\rm SU}(2)(\mathbb{C}))$ is trivial.  Again $L$ is the
only Picard--Vessiot field.\\
(2). If $G$ is the symplectic group ${\rm Sp}(2n)_{k}$, then  $H^1(\{1,\sigma \},G(k(i)))$ is trivial. Therefore there is only one Picard--Vessiot field $L$ and this is a real field.  \\ 
(3). Consider a $k$-form $G$ of ${\rm SO}(n)_{k}$ with odd $n\geq 3$. The center $Z$ of $G$ consists of the scalar matrices of order $n$, thus $Z$ is the group $\mu_{n,k}$ of the $n$th roots of unity. Since $n$ is odd, one has $Z(k)=\{1\}$. Further, again since $n$ is odd, the automorphisms of $H=G_{k(i)}$ are inner and $Aut(H)(k(i))=G/Z(k(i))$. We claim the following. \\

\noindent 
{\it The natural map $H^1(\{1,\sigma \},G(k(i)) )\rightarrow
H^1(\{1,\sigma \},G/Z(k(i)) )$ is a bijection}.\\

\noindent 
{\it Proof}. A 1-cocycle $c$ for $G/Z(k(i)) $ is given by $c(1)=1$ and
$c(\sigma )=a\in G/Z(k(i)) $ with $a\sigma (a)=1$. Choose an
$A\in G(k(i))$ which maps to $a$. Thus $A\sigma (A)\in Z(k(i))$ and
$A$ commutes with $\sigma A$. Further $\sigma (A\sigma (A))=
\sigma (A)A=A\sigma (A)$ and thus $A\sigma (A)\in \mu _n(k)=\{1\}$.
Therefore $C$ defined by $C(1)=1,\ C(\sigma )=A$ is a 1-cocycle for
$G(k(i))$ and maps to $c$. Hence the map is surjective.

Consider for $j=1,2$, the 1-cocycle $C_j$ for $G$ given by 
$C_j(\sigma )=A_j$. Suppose that the images of $C_j$ as 1-cocycles for $G/Z(k(i))$ are equivalent. Then there exists  $B\in G(k(i))$ such that $B^{-1}A_1\sigma (B)=xA_2$ for some
element $x\in Z(k(i))$. We may replace $B$ by $yB$ with $y\in Z(k(i))$.
Then $x$ is changed into $xy^{-1}\sigma (y)$. And the latter is equal
to 1 for a suitable $y$. This proves the injectivity of the map.\hfill $\square$\\

We conclude from the above result  that there exists a (unique up to isomorphism) fibre functor $\eta :<M>_\otimes \rightarrow vect(k)$
(or, equivalently, a Picard--Vessiot field) for every form of $H=SO(n)_{k(i)}$ over $k$. Moreover, only one of these
fibre functors  corresponds to a {\it real} Picard--Vessiot field.

\bigskip

Let $\omega : <M>_\otimes \rightarrow vect(k)$ denote the fibre functor
corresponding to a {\it real} Picard--Vessiot field $L_\omega$ and 
$G_\omega$ the group of the differential automorphisms of $L_\omega/K$. 
{\it We want to  identify this form $G_\omega$ of 
$H:={\rm SO}(n)_{k(i)}$}.\\

Since the differential Galois group of $K(i)\otimes M$ is ${\rm SO}(n)_{k(i)}$, there exists an element $F\in sym^2(K(i)\otimes M^*)$ with
$\partial F=0$. Further $F$ is unique up to multiplication by a scalar and
$F$ is a non degenerate bilinear symmetric form. The non trivial automorphism $\sigma$ of $K(i)/K$ and of $k(i)/k$ acts in an obvious way  on $K(i)\otimes M$ and  on constructions by linear algebra
of  $K(i)\otimes M$. Now $\sigma (F)$ has the same properties as $F$ and
thus $\sigma (F)=cF$ for some $c\in K(i)$. After changing $F$ into 
$aF$ for a suitable $a\in K(i)$, we may suppose that $\sigma (F)=F$. 
Then $F$ belongs to $sym^2(M^*)$ and is a non degenerate form of degree $n$ over the field $K$. Further $F$ is determined by its signature because $K$ is real closed. Moreover $KF$ is the unique 
1-dimensional submodule of $sym^2(M^*)$. We claim the following:\\
  {\it
$G_\omega$ is the special orthogonal group over $k$ corresponding to
a form $f$ over $k$ which has the same signature as $F$}. \\

Let $V=\omega (M)$. The group 
$G_\omega$ is the special
orthogonal group of some non degenerate bilinear symmetric form 
$f\in sym^2(V^*)$. Since $L_\omega$ is real, there exists an  isomorphism
$m:K\otimes _k\omega \rightarrow \rho$ of functors. Applying $m$
to the modules $M$ and  $sym^2(M^*)$ one finds an isomorphism 
$m_1:K\otimes _kV\rightarrow M$ of $K$-vector spaces 
which induces an isomorphism of $K$-vector spaces 
$m_2:K\otimes_k sym^2(V^*)\rightarrow sym^2(M^*)$.
The latter maps the subobject $K\otimes kf$ to $KF$ by the uniqueness of
$KF$.   One concludes that the forms $f$ and $F$ have the same signature.\hfill $\square$
}\end{examples}

\begin{proposition} Suppose that $K$ is real closed.
 Given is a connected semi-simple group $H$ over $k(i)$ and a form $G$ of $H$ over $k$. Then there exists a differential module $M$ over $K$ and a real Picard--Vessiot field for $M/K$ such that the group of the differential automorphisms of $L/K$ is $G$.\end{proposition}

 \begin{proof} Let $G$ be given as a subgroup of some ${\rm GL}_{n,k}$, defined by
 a radical ideal $I$. Then $k[G]=k[\{X_{k,l}\}_{k,l=1}^n,\frac{1}{\det }]/I$. The tangent space of
 $G$ at $1\in G$ can be identified with the $k$-linear derivations $D$ of this algebra, commuting with the action of $G$.  These derivations $D$ have the form $(DX_{k,l})=B\cdot (X_{k,l})$ for some
matrix $B\in Lie(G)(k)$ (where $Lie(G)\subset {\rm Matr}(n,k)$ is the Lie algebra of $G$).

The same holds for $K[G]=K\otimes _k k[\{X_{k,l}\}_{k,l=1}^n,\frac{1}{\det }]/I$. Any
$K$-linear derivation $D$ on the algebra, commuting with the action of $G$, has the form
$(DX_{k,l})=A\cdot (X_{k,l})$ with $A\in Lie(G)(K)$. We choose $A$ as general as possible.

The differential module $M/K$ is defined by the matrix equation
$y'=Ay$. It follows from \cite{PS}, Proposition 1.3.1 that the differential Galois group of $K(i)\otimes M$ is contained in $H=G_{k(i)}$.  Now one has to choose $A$ such that the differential Galois group (which is connected because $K(i)$ is algebraically closed) is not a proper subgroup of $H$. Since $H$ is semi-simple, there exists a
Chevalley module for $H$. Using this Chevalley module one can produce
a general choice of $A$ such that the differential Galois
group of $y'=Ay$ over $K(i)$ is in fact $G_{k(i)}$ (compare \cite{PS}, \S 11.7 for the details which remain valid in the present situation).  

The usual way to produce a Picard--Vessiot ring for the
equation  $y'=Ay$ is to consider the differential algebra
$R_0:=K[\{X_{k,l}\}_{k,l=1}^n,\frac{1}{\det }]$, with
differentiation defined by $(X'_{k,l})=A\cdot (X_{k,l})$, and to
produce a maximal differential ideal in  $R_0$.  Since 
$A\in Lie(G)(K)\subset Lie(H)(K(i))$, the ideal $J\subset 
R_0[i]$, generated by $I$ is a differential ideal.  It is in fact a maximal 
differential ideal of $R_0[i]$, since the differential Galois group is precisely $H$.
Then $J\cap R_0=IR_0$ is a maximal differential ideal of $R_0$ and 
$K[G]=R=R_0/IR_0$  is a Picard--Vessiot ring for $M$ over $K$.
The field of fractions $L$ of $R$ is real because the $G$-torsor
$Spec(K[G])$ is trivial.   \end{proof}

It seems that, imitating the proofs in \cite{MS}, one can show that
Proposition 3.2  remains valid under the weaker conditions: $K$ is a real differential field and a $C_1$-field and $H$ is connected.

\end{document}